\documentclass[]{article}
\usepackage{indentfirst}
\usepackage{amsmath}
\usepackage{authblk}
\usepackage{geometry}
\usepackage{accents}
\usepackage{statex}
\usepackage[francais,english]{babel}
\usepackage{url}

\usepackage[utf8]{inputenc} 
\usepackage[nottoc]{tocbibind}

\newtheorem{theorem}{Theorem}
\newtheorem{proposition}{Proposition}
\newtheorem{lemma}{Lemma}
\newtheorem{corollary}{Corollary}

\newenvironment{proof}[1][Proof]{\begin{trivlist}
		\item[\hskip \labelsep {\bfseries #1}]}{\end{trivlist}}

\newenvironment{remark}[1][Remark]{\begin{trivlist}
		\item[\hskip \labelsep {\bfseries #1}]}{\end{trivlist}}

\title{ Palindromic length complexity and a generalization of Thue-Morse sequences}
\author{Shuo LI}
\date {}

\begin{document}
	
\maketitle
\section {Introduction}
The notion of palindromic length of a finite word  as well as an infinite word was first introduced by Frid,  Puzynina and Zamboni\cite{FRID2013737}. They conjectured that if the palindromic length of an infinite word is bounded, then this sequence is eventually periodic. This conjecture is widely studied by \cite{FRID2013737}\cite{FRID2018202}\cite{ambroz}, and the palindromic length of some specific sequences are studied as well: Frid \cite{FRID2018202} showed that Sturmian  words  have an unbounded palindromic  length $PL_u$  and Ambro\v{z} \cite{ambroz} showed that $PL_u$ grows arbitrarily slowly. \cite{AMBROZ2019} studied palindromic lengths of fixed points of a specific class of morphisms and gave upper bounds for the Fibonacci word and the Thue-Morse word. In this article, we give a formal expression of the palindromic  length of Thue-Morse sequence and find all sequences which have the same palindromic  length as Thue-Morse's.  After writing a first version of this paper, we found that some results in the same direction were obtained by Frid \cite{anna} for Thue-Morse sequences. However, we will indicate how our results can be applied for a type of generalization of Thue-Morse sequences.  

\section {Definitions and notation}

Let $(a_n)_{n\in \mathbf{N}}$ be a sequence and  let us define a (finite) word, or a factor, of a sequence to be a (finite) string of the sequence. Let $w_a(x, y)$ denote the factor of the sequence $(a_n)_{n\in \mathbf{N}}$ beginning at the position $a_x$ of length $y$, in other words $w_a(x,y)=\overline{a_xa_{x+1}...a_{x+y-1}}$. \\

Let $\widetilde{w}$ denote the reversal of $w$, that is to say, if $w=\overline{w_0w_1...w_k}$ then $\widetilde{w}=\overline{w_kw_{k-1}...w_0}$, we say a word $w$ is palindromic if $w=\widetilde{w}$. Let us denote by $Pal$ the set of all palindromic words. \\

We define the palindromic length of a word $w$, which will be denoted by $|w|_{pal}$, to be:

$$|w|_{pal}=\min\left\{k|w=p_1p_2...p_k, p_i \in Pal,\  \forall i \in [1,k]\right\},$$
in this case we say $w=p_1p_2...p_k, p_i$ is an optimal palindromic decomposition of $w$.\\

Let us define the palindromic length sequence $(pl_a(n))_{n\in \mathbf{N}}$ of the sequence $(a_n)_{n\in \mathbf{N}}$ to be 

$$pl_a(n)=|w_a(0,n)|_{pal},$$
in other words, $pl_a(n)$ is the palindromic length of the word $\overline{a_0a_1...a_{n-1}}.$\\

Now let us define a class of infinite sequences $\mathcal{C}$ which can be considered as an generalization of the Thue-Morse sequence:\\
Let $\sum$ be an alphabet which contains at least two letters and let $a \in \sum$.\\
Let $F$ be the set of bijections over $\sum$. \\
Let $(f_n)_{n \in \mathbf{N}}$ be a sequence over $F$ and $(w_n)_{n \in \mathbf{N}}$ be a sequence of finite words over $\sum$ which are defined recursively as:
$$ f_i \in F \;\text{such that}\; f_i(w_i)\neq w_i  \forall n\geq0,$$
and
$$\begin{cases}
    w_0=a \\
    w_n=w_{n-1}f_{n-1}(w_{n-1})f_{n-1}(w_{n-1})w_{n-1}, \forall n>0.
  \end{cases}$$
Let $f(a)$ be the limit of the sequence $(w_n)_{n \in \mathbf{N}}$ which exists because of the definition.\\
The class $\mathcal{C}$ is the set of all infinite limits defined as above. It is easy to see that, if the size of $\sum$ is equal to $2$, say $\sum=\left\{a,b\right\}$, then all sequences in $\mathcal{C}$ are Thue-Morse sequences, they may be written as
$$a,b,b,a,b,a,a,b, b,a,a,b,a,b,b,a...$$ and the one by changing $a$ to $b$ and $b$ to $a$.
Let $(pl(n))_{n\in \mathbf{N}}$ be the palindromic length sequence of Thue-Morse, the first elements of this sequence are

$$1,2,2,1,2,3,3,2,3,4,3,2,3,3,2,1...$$

\section{Palindromic length of sequences in $\mathcal{C}$} 

In this section we will study palindromic lengths of sequences in $\mathcal{C}$ and prove that they all have the same palindromic length, as the one of Thue-Morse. \\

Let $(a_n)_{n\in \mathbf{N}}$ be a sequence in $\mathcal{C}$, we will begin with some properties of palindromic factors of this sequence. 

\begin{lemma}
For any integer $x$, $w_a(4x, 4)$ is of type $abba$ such that $a \neq b$.\\
As a corollary, $a(2n+1)\neq a(2n)$. 
\end{lemma}

\begin{proof}
This lemma is trivial because of the definition. 
\end{proof}

\begin{lemma}
Let $w_a(x, y)$ be a palindromic factor of Thue-Morse sequence such that $y$ is odd, then $y$ is either $1$ or $3$.
\end{lemma}

\begin{proof}
If $w_a(x, y)$ is of size larger than $3$, then it contains at least one palindromic word in the center of size $5$, however a word of size $5$ should be inside of a word of type $xyyxxyyx$ or $xyyxf(x)f(y)f(y)f(x)$,where $f$ is a bijection over the alphabet defined as above, but none of them contains a palindromic word of such size. 
\end{proof}

\begin{lemma}
Let $w_a(x, y)$ be a palindromic word of $(a_n)_{n\in \mathbf{N}}$ such that $y$ is even, then either there exist $z, r\in \mathbf{N}$ such that $w_a(x, y)$ is embedded into the center of palindromic word $w_a(4z, 4r)$ or $x \equiv 3 \mod 4$ and $y \equiv 2 \mod 4$.
\end{lemma}

\begin{proof}
We first prove that $x+y/2-1$ is odd, otherwise $x+y/2-1=2t$ and $x+y/2=2t+1$ for some $t$, so that $a_{2t}=a_{2t+1}$ contradicts to Lemma 1. This fact implies that 
$$\begin{cases}
    \text {if}\; x \equiv 0 \mod 4\; \text{then}\; x+y-1 \equiv 3 \mod 4 \\
    \text {if}\; x \equiv 1 \mod 4\; \text{then}\; x+y-1 \equiv 2 \mod 4 \\
    \text {if}\; x \equiv 2 \mod 4\; \text{then}\; x+y-1 \equiv 1 \mod 4 \\
    \text {if}\; x \equiv 3 \mod 4\; \text{then}\; x+y-1 \equiv 0 \mod 4.\\
  \end{cases}.$$
For the last case, we have $x \equiv 3 \mod 4$ and $y \equiv 2 \mod 4$. Now let us check that, for other cases, the word $w_a(x, y)$ can be embedded into the center of a palindromic word of type $w_a(4z, 4r)$. Let $w_a(4z, 4r)$ be the shortest factor of above type including $w_a(x, y)$, here we prove that this factor is palindromic. It is easy to see that $w_a(x, y)$ is at the center of $w_a(4z, 4r)$  and the word $w_a(4(z+1), 4(r-1))$ is palindromic because of the palindromicity of $w_a(x, y)$; furthermore we have the fact that $w_a(4z, 4)=\widetilde{w_a}(4(z+r-1), 4)$ when $x\not \equiv 3 \mod 4$, because these two words of length $4$ are both palindromic and uniquely defined by respectively a prefix or a suffix of $w_a(x, y)$ of size smaller than $4$ but larger than $1$. In conclusion, the word $w_a(4z, 4r)$ is palindromic.
\end{proof}

\begin{lemma}
Let $w_a(0, s)$ be a prefix of $(a_n)_{n\in \mathbf{N}}$ in $\mathcal{C}$, and let $w_a(0, s)=p_1p_2...p_r$ be an optimal palindromic decomposition such that for all $i: 1 \leq i \leq r$, $p_i$ is either singleton or can be embedded into the center of palindromic word of type $w_a(4z, 4t)$, then there exists at least one optimal palindromic decomposition of $w_a(0,s)$ of following forms:
$$\begin{cases}
w_a(0, s)=q_1q_2...q_r; s \equiv 0 \mod 4;\\
w_a(0, s)=q_1q_2...q_{r-1}t_1; s \equiv 1 \mod 4;\\
w_a(0, s)=q_1q_2...q_{r-2}t_1t_2; s \equiv 2 \mod 4;\\
w_a(0, s)=q_1q_2...q_{r-2}t_1l_1; s \equiv 3 \mod 4;\\
w_a(0, s)=q_1q_2...q_{r-3}t_1t_2l_1; s \equiv 2 \mod 4;
\end{cases}$$
where $q_i$ are palindromes of length $4k_i$, $t_i$ are singletons and $l_i$ are palindromes of length $2p_i$.

\end{lemma}

\begin{proof}

Let us consider a factor of $(a_n)_{n\in \mathbf{N}}$ of type $rq_1q_2..q_{2l}$ where $1 \leq |r| \leq 2$ beginning at some position $4x$ where $q_i$ are palindromic words of even size and can be embedded into the center of palindromic word of type $w_a(4z, 4r)$. Here we prove that there exists an other palindromic decomposition of same length such that 
$$rq_1q_2..q_{2l}=q'_1q'_2..q'_{2l}r,$$
where all $q'_i$ are of size $4k_i$. \\

As $q_1$ is palindromic, because of Lemma 3, $rq_1\widetilde{r}$ is also palindromic, let us denote such word by $q'_1$, its size is multiple of $4$. By excluding the case that $|r|=|q_2|=2$, $q_2$ can be written as $\widetilde{r}q'_2r$, where $q'_2$ is either a palindromic word of size $4m$ or empty, so we have the equality $rq_1q_2=q'_1q'_2r$ and the last $r$ begins at some position $4x$. We do it recursively and we end up with the expression $rq_1q_2..q_{2l}=q'_1q'_2..q'_{2l}r$.

In such a way we can accumulate the singletons in the decomposition $w_a(0, s)=p_1p_2...p_r$ and push them to the end. An easy  observation is that there are at most two singletons in an optimal decomposition, since once there are three singletons , they will meet each other by the above algorithm in a block $w_a(4k,4l)$ hence two of them will create a palindromic word of length $2$ which contradicts the optimality. The above process ends up with five possibilities:
$$\begin{cases}
w_a(0, s)=q_1q_2...q_r;\\
w_a(0, s)=q_1q_2...q_{r-1}t_1; \\
w_a(0, s)=q_1q_2...q_{r-2}t_1t_2; \\
w_a(0, s)=q_1q_2...q_{r-2}t_1l_1; \\
w_a(0, s)=q_1q_2...q_{r-3}t_1t_2l_1;
\end{cases}$$
where $q_i$ are palindromes whose length are multiple of $4$, $t_i$ are singletons and $l_i$ are palindromes whose length are multiple of $2$.

The first case leads to $s \equiv 0 \mod 4$; second one leads to $s \equiv 1 \mod 4$ and the third one leads to $s \equiv 2 \mod 4$; for the fourth one we can check that $|l_1| \not \equiv 0 \mod 4$ because of Lemma 3, so that $s \equiv 3 \mod 4$; the fifth case, $|l_1|$ must be a multiple of $4$, so $s \equiv 2 \mod 4$.

\end{proof}

\begin{corollary}
Let  $(pl(n))_{n\in \mathbf{N}}$ be the palindromic length of a sequence in $\mathcal{C}$ such that all its prefixes admit an optimal palindromic decomposition satisfying the constrains listed as in the previous lemma,  then for all $k\geq 0$:\\
$pl(4k+i)\geq pl(4k+3) +1$ for $i=1,2$ and $pl(4k)\geq pl(4k+3) $.
\end{corollary}

\begin{proof}
For $i=0$, $w_a(0, 4k+1)$ is of the form $q_1q_2...q_rt_1$.  Using Lemma 1 we have $w_a(0, 4k+4)=q_1q_2...q_rq$ is a palindromic decomposition, not necessarily optimal, with $q=w_a(4k,4)$, so $pl(4k+3) \leq r+1=pl(4k)$

For $i=1$, there are $2$ cases: if $w_a(0, 4k+2)$ is of the form $q_1q_2...q_rt_1t_2$,  then once more using Lemma 1 we have $w_a(0, 4k+4)=q_1q_2...q_rq$ is a palindromic decomposition, with $q=w_a(4k,4)=t_1t_2t_2t_1$; if $w_a(0, 4k+2)$ is of the form $q_1q_2...q_rt_1t_2l_1$, using the hypothesis we have $w_a(0, 4k+4)=q_1q_2...q_rq$ is a palindromic decomposition, with $q=t_1t_2l_1t_2t_1$.

For $i=2$, $w_a(0, 4k+3)=q_1q_2...q_rt_1l_1$, using the hypothesis we have $w_a(0, 4k+4)=q_1q_2...q_rq$ is a palindromic decomposition, with $q=t_1l_1t_1$.

So all inequalities as above are proved.

\end{proof}

\begin{lemma}
Let $w_a(0, k)$ be a prefix of $(a_n)_{n\in \mathbf{N}}$, then there is an optimal decomposition $w_a(0, k)=p_1p_2...p_s$ such that none of these palindromes is of length $3$, furthermore, if $p_i$ is of even size then it can be embedded into the center of palindromic word of type $w_a(4z, 4r)$.
\end{lemma}

\begin{proof}
Let us suppose that $k$ is the smallest number such that $w_a(0, k)$ does not satisfy one of the two constrains above, then either the last palindromic factor in all optimal compositions is of length $3$, or it can not be embedded into the center of palindromic word of type $w_a(4z, 4r)$. If it is in the first case, then the last factor can be found either at position $\overline{a_{4t-1}a_{4t}a_{4t+1}}$ or $\overline{a_{4t-2}a_{4t-1}a_{4t}}$. If $k=4t+2$, then optimal decompositions of $w_a(0,k)$ are of the form $w_a(0, 4t+2)=w_a(0, 4t-1)p$, so that the palindromic length is $pl(4t+1)=1+pl(4t-2)$, otherwise, if we decompose the word as $w_a(0, 4t+2)=w_a(0, 4t)a_{4t}a_{4t+1}$, we have a length $pl(4t-1)+2$, so that
$$pl(4t+1)=1+pl(4t-2) < pl(4t-1)+2;$$
similarly for the case that $k=4t+1$, by considering the decomposition $w_a(0, 4t+1)=w_a(0, 4t)a_{4t}$, we have
$$pl(4k)=1+pl(4k-3) < pl(4k-1)+1,$$ 
both inequalities contradict the previous corollary.

If the last factor can not be embedded into the center of a palindromic word of type $w_a(4z, 4r)$, then because of Lemma 3 it can be found at some position $\overline{a_{4t-1}a_{4t}...a_{4l}}$, so the optimal decomposition is $w_a(0, 4l)=w_a(0, 4t-1)\overline{a_{4t-1}a_{4t}...a_{4l}}$. However, if we consider another composition $w_a(0, 4l)=w_a(0, 4t)\overline{a_{4t}a_{4t}...a_{4l-1}}a_{4l}$, we have
$$pl(4t-2)+1 < pl(4t-1)+2,$$ contradicts the previous corollary.

\end{proof}

\begin{corollary}
Corollary 1 is valid for all sequences in $\mathcal{C}$.
\end{corollary}

\begin{lemma}
The palindromic length sequence $(pl(n))_{n\in \mathbf{N}}$ satisfies for $k\geq 0$:
$$pl(4k+i) \leq pl(4k+3)+2,$$ when $i=0$ or $1$; and
$$pl(4k+2) \leq pl(4k+3)+1.$$

\end{lemma}

\begin{proof}
Let $w_a(0, 4k+4)=p_1p_2...p_s$ be an optimal palindromic decomposition such that all $p_j$ are of size $4r_j$ which exists because of Lemma 5.

If the size of $p_s$ is larger than $4$, then for $i=1,2$ or $3$, we can write $p_s=ab\widetilde{a}$ where $a$ is the prefix of $p_s$ of length $4-i$ so in this case$$w_a(0, 4k+i)=p_1p_2...p_{s-1}ab$$ and $|a|_{pal}=2$ when $i=3$ and $|a|_{pal}=1$ otherwise.

If the size of $p_s$ is $4$, then for $i=1,2$ or $3$, we can write $p_s=ab$ where $a$ is the prefix of $p_s$ of length $i$ so in this case$$w_a(0, 4k+i)=p_1p_2...p_{s-1}a$$ and $|a|_{pal}=1$ when $i=1$ and $|a|_{pal}=2$ otherwise.
In both cases the above inequalities hold.
\end{proof}

\begin{lemma}
Let $(a_n)_{n \in \mathbf{N}}$ be a sequence in $\mathcal{C}$ defined over the alphabet $\sum$, let $f$ be a bijection from $\sum^4$ to a new alphabet $\sum'$, then the sequence $(b_n)_{n \in \mathbf{N}}$ defined as 
$$b_n=f(\overline{a_{4n}a_{4n+1}a_{4n+2}a_{4n+3}}) \forall n \in\mathbf{N}$$
 is also in $\mathcal{C}$. As a consequence, $w_a(0,4t)=p_1p_2...p_k$ is an optimal palindromic decomposition of $w_a(0,4t)$ if and only if $w_b(0,t)=f(p_1)f(p_2)...f(p_k)$ is an optimal palindromic decomposition of $w_b(0,t)$ and the palindromic length sequence $(pl(n))_{n\in \mathbf{N}}$ satisfies for $k\geq 0$:
$$pl(4k+3) = pl(k)$$
\end{lemma}

\begin{proof}
The first part is easy to check by induction. For the second part, applying the algorithm introduced in Lemma 3 to $w_a(0, 4k+4)$, we get an optimal decomposition such that all palindromic words in the optimal decomposition are of size $4k_i$ and begin at some position $4r_i$. Applying $f$ to $w_a(0, 4k+4)$ as well as each palindromic factor, we get a decomposition of a word of length $k+1$, which is a prefix of the sequence $(b_n)_{n \in \mathbf{N}}$, this decomposition is optimal because of the bijectivity of $f$.
\end{proof}

\begin{corollary}
The palindromic length sequence $(pl(n))_{n\in \mathbf{N}}$ satisfies for $k\geq 0$:\\
$pl(4k+3) = pl(k)$;\\
$pl(4k+2) = pl(4k+3)+1$;\\
$pl(4k+1) = pl(4k+3)+1$ or $pl(4k+3)+2$;\\
$pl(4k) = pl(4k+3)$, $pl(4k+3)+1$ or $pl(4k+3)+2$.\\
\end{corollary}

\begin{proposition}
The palindromic length sequence $(pl(n))_{n\in \mathbf{N}}$ satisfies for $k\geq 0$:\\
$pl(4k+1) = pl(4k+3)+1$ if $k \equiv 0 \mod 4$;\\
$pl(4k+1) = pl(4k+3)+2$ if $k \equiv 2,3 \mod 4$;\\
$pl(4k+1) - pl(4k+3)= pl(k+1) - pl(k+3)$ if $k \equiv 1 \mod 4$;\\
$pl(4k) = pl(k-1)+1$.
\end{proposition}

\begin{proof}
If $k \equiv 0 \mod 4$, applying the bijection introduced in Lemma 7, the optimal decomposition of $w_a(0,4k+4)$ is $w_a(0,4k)\overline{a_{4k}a_{4k+1}a_{4k+2}a_{4k+3}}$, so that 
$w_a(0,4k+2)=w_a(0,4k)\overline{a_{4k}a_{4k+1}}$ is a decomposition of $w_a(0,4k+2)$. As a result, $pl(4k+1) \leq pl(4k-1)+2=pl(4k+3)+1$.\\

If $k \equiv 2,3 \mod 4$, it is enough to prove that the last factor in any optimal palindromic decompositions of $w_a(0,4k+4)$ is of length larger than $4$. This is trivial by applying the bijection $f$ to $w_a(0,4k+4)$ and concluding by the classification in Lemma 4.\\

If $k \equiv 1 \mod 4$, applying the bijection introduced in Lemma 7 and Lemma 4, the optimal decomposition of $w_a(0,4k+4)$ is either of type $p_1p_2...p_kt_1t_2$  or of type $p_1p_2...p_kt_1t_2l$, with $p_i$ and $l$ of length $16r_i$ and $t_i$ of length $4$. The first case implies $pl(4k+1) - pl(4k+3)=1$ while the second case implies $pl(4k+1) - pl(4k+3)=2$. However, if we apply $f$ to $w_a(0,4k+4)$ we get a word of length $k+1$ and $pl(k+1) - pl(k+3)=1$ in the first case and $pl(k+1) - pl(k+3)=2$ in the second case.\\

The last equality is a consequence of Lemma 4 and Lemma 7.
\end{proof}

\begin{proposition}
All sequences in $\mathcal{C}$ share the same palindromic length sequence $(pl(n))_{n\in \mathbf{N}}$. Furthermore, it is $4$-regular.
\end{proposition}

\begin{proof}
The $4$-kernel of $(pl(n))_{n\in \mathbf{N}}$ is generated by elements in  
$$\left\{(pl(n))_{n\in \mathbf{N}},(pl(n-1))_{n\in \mathbf{N}},(pl(n+1))_{n\in \mathbf{N}},(pl(n+3))_{n\in \mathbf{N}},(1)_{n\in \mathbf{N}}\right\}$$
\end{proof}

\begin{remark}
Lemma 7 and Proposition 1 are critical in the proof because they show the importance of the hypothesis that $f_n(w_n) \neq w_n$. Because of this hypothesis, we can guarantee that the set $\mathcal{C}$ is closed under bijections (and their inverses) defined in Lemma 5, and do not have factors like $aaaa$ in the sequence. So that we can apply some inductive properties by saying that $w_a(0,4k+4)$ and $w_b(0,k+1)$ share the ``same" optimal palindromic decomposition, which is the key point to make Proposition 1 work.
\end{remark}

\begin{corollary}
 $pl(n)+1 \geq pl(n+1)$;\\
 if there exists an integer $n$ satisfying $pl(n)+2=pl(n+1)+1=pl(n+2)$, then $n \equiv 3\mod 4$;\\
 if $pl(4k)=pl(4k+3)$ then $pl(4k+1)=pl(4k+2)=pl(4k+3)+1$;\\
 if $pl(4k)=pl(4k+1)$ then $pl(4k)=pl(4k+1)=pl(4k+3)+2$.
\end{corollary}

\begin{proof}
The first statement is trivial because of a decomposition $w_a(0,n+2)=w_a(0,n+1)a_{n+1}$.\\

For the second statement, remarking the fact that $pl(4k+3)=pl(4k+2)-1$, we have either $n \equiv 3 \mod 4$ or $n+3 \equiv 3 \mod 4$, but if it is the last case, then $pl(n)+2=pl(n+2)=pl(n+3)+1$ so that $pl(n+3) > pl(n)$ which contradicts Corollary 3.\\

For the last two statements, $pl(4k)=pl(4k+3)$ implies that the last palindromic factor in optimal decompositions of $w_a(0,4k+4)$ is  $\overline{a_{4k}a_{4k+1}a_{4k+2}a_{4k+3}}$ which proves $pl(4k+1)=pl(4k+2)=pl(4k+3)+1$. On the contrary, if $pl(4k)=pl(4k+1)$ then $pl(4k) \neq pl(4k+3)$, so that the last palindromic factor in optimal decompositions of $w_a(0,4k+4)$ is of length larger than $4$, which leads to the fact $pl(4k)=pl(4k+1)=pl(4k+3)+2$.
\end{proof}

\section{All sequences sharing $(pl(n))_{n\in \mathbf{N}}$}

In this section, we are going to prove that all sequences sharing the same palindromic length $(pl(n))_{n\in \mathbf{N}}$ defined in the previous section are exactly the functions in $\mathcal{C}$.

\begin{lemma}
Let $(b_n)_{n\in \mathbf{N}}$ be a sequence such that all words $w_b(4k,4)$ are of form $xyyx$, then\\ 
1) if $w_b(a,b)$ is a palindromic word and $b$ is odd, then $b \leq3$, furthermore, if $b=3$, then $a \equiv 3,0 \mod 4$.\\
2) if $w_b(a,b)$ is a palindromic word and $b$ is even, then $a+b/2-1$ is odd.
\end{lemma}

\begin{proof}
It is analogous to Lemma 2 and Lemma 3.
\end{proof}

\begin{lemma}
Let $(b_n)_{n\in \mathbf{N}}$ be a sequence such that its palindromic length sequence coincides with $(pl(n))_{n\in \mathbf{N}}$, then all words $w_b(4k,4)$ are of form $xyyx$ with $x \neq y$.
\end{lemma}

\begin{proof}
We prove the statement by induction:\\

Firstly the statement holds for $s=0$. Supposing that this statement is true for all $s \leq s_0$, we will prove it for $s = s_0+1$.\\

Let us consider a decomposition $w_b(0, 4s_0+4)=p_1p_2...p_r$ such that $r=pl(4s_0+4)$, and let us denote by $n$ the length of $p_r$. \\
Firstly $n$ can not be too small: if $n<4$ then $pl(4s_0+3)=1+pl(4s_0+3-n)>pl(4s_0+3)$ which contradicts Corollary 2. \\
Secondly, if $n$ is odd then it can not be too large: if $n=2n_0+1$ and $n_0 > 4$ then $w_b(4s_0+6-2n_0,2n_0-7)$ is a palindrome of odd size larger or equal to $3$ and finishing at the position $4s_0-1$, which does not exist because of the Lemma 8. \\
Thirdly, if $n$ is even and large enough: if $n=2n_0$ and $n_0 \geq 4$, then, because of Lemma 8, $n$ is a multiple of $4$ and $w_b(4s_0-1,4)$ is the inverse of some words $xyyx$.\\
So there are $5$ other cases to study: $n=4,5,6,7,9$.\\

When $n=4$, $w_b(4s_0,4)$ is either of type $xxxx$ or $xyyx$, and $pl(4s_0-1)+1=pl(4s_0+3)$, if $w_b(4s_0,4)$ is of type $xxxx$ then $w_b(0,4s_0+3)=w_b(0,4s_0)xxxx$ so $pl(4s_0+2)\leq pl(4s_0-1)+1=pl(4s_0+3)$, contradicts Corollary 2.\\

When $n=5$ or $6$, $pl(4s_0+3)=pl(4s_0+3-n)+1 > pl(4s_0-1)+1$, however, $pl(4s_0+3) \leq pl(4s_0-1)+1$, contradiction.\\

When $n=7$, $pl(4s_0+3)=pl(4s_0-4)+1= pl(4s_0-5)+2$. On the other hand, $pl(4s_0-1) \leq pl(4s_0-5)+1$ and $pl(4s_0+3) \leq pl(4s_0-1)+1$, so $$pl(4s_0+3) = pl(4s_0-1)+1= pl(4s_0-5)+2. \; (*)$$ After Corollary 4, $4s_0-4 \equiv 0 \mod 16$ and 
$$pl(4s_0+7)= pl(4s_0+3)\; \text{or}\; pl(4s_0+7)= pl(4s_0+3)-1 \; (**).$$ If we write $w_b(0,4s_0+4)=w_b(0,4s_0-4)abbaxabb$ let us consider the last palindromic factor of $w_b(0,4s_0+8)$:\\
1) The length can not be smaller than $4$, otherwise $pl(4s_0+7)=pl(4s_0+i)+1>pl(4s_0+7)$ with $3 <i <6$, contradicts Corollary 3.\\
2) The length can not be $4,5,6,7$, otherwise $pl(4s_0+7)=pl(4s_0+i)+1$ with $-1 <i \leq3$, but $pl(4s_0+i)\geq pl(4s_0+3)$, so that $pl(4s_0+7)>pl(4s_0+3)+1$, contradicts $(**)$.\\
3) The length can not be $8$, otherwise $w_b(0,4s_0+8)=w_b(0,4s_0-4)abbaxabbbbax$ and $pl(4s_0+7)=pl(4s_0-1)+1=pl(4s_0+3)$. But on the other hand, $pl(4s_0+4)=pl(4s_0+3)+1$ and $pl(4s_0+6)=pl(4s_0+7)+1=pl(4s_0+3)+1$ because of Proposition 1;  $pl(4s_0+5)=pl(4s_0+3)+1$ because of the decomposition $w_b(0,4s_0+6)=w_b(0,4s_0+4)bb$ so that $pl(4s_0+4)=pl(4s_0+5)=pl(4s_0+6)$ which contradicts Corollary 4.\\
4) The length can not be $9,10$, otherwise  $pl(4s_0+7)=pl(4s_0-i)+1$ with $i=2,3$, but $pl(4s_0-i) \geq pl(4s_0-1)+1$ so that $pl(4s_0+7)>pl(4s_0-1)+1=pl(4s_0)+3$, contradicts $(**)$.\\
5) The length can not be $11,12,13,14,16,17$, because the last factor can not be palindromic.\\
6) The length can not be $15$, otherwise, $w_b(0,4s_0+8)=w_b(0,4s_0-8)cddcabbaxabbacdd$, with $a \neq b, c\neq d$. Let us check a decomposition $w_b(0,4s_0+5)=w_b(0,4s_0-4)abbaxabba$, so that $pl(4s_0+4)\leq pl(4s_0-5)+1$, but $pl(4s_0+4)=pl(4s_0+3)+1$ which implies $pl(4s_0+3) \leq pl(4s_0-5)$, contradicts $(*)$.\\
7) The length can not be an odd number larger than $15$, otherwise, there is a palindromic factor of odd size larger than $3$ in $w_a(0,4s_0)$ finishing at position $4s_0-1$, contradicts to Lemma 8.\\
8) The length can not be an even number larger than $14$, otherwise,because of Lemma 8, the length is a multiple of $4$, which implies the factor $w_a(4s_0,4)$ is the symmetry of some words $w_a(4x,4)$, by hypothesis, it is of type $abba$ but not $xabb$.\\
In conclusion, the last palindromic factor of $w_a(0,4s_0+4)$ can not be $7$.\\

When $n=9$, $pl(4s_0+3)=pl(4s_0-6)+1= pl(4s_0-5)+2 \geq pl(4s_0-1)+1$. On the other hand, $pl(4s_0+3) \leq pl(4s_0-1)+1$, so $pl(4s_0+3) = pl(4s_0-1)+1$; another observation is that $pl(4s_0+2) \leq pl(4s_0-5)+1$ because $\overline{b_{4s_0-4},b_{4s_0-3},b_{4s_0-2},b_{4s_0-1},b_{4s_0},b_{4s_0+1},b_{4s_0+2}}$ is palindromic, but $pl(4s_0+2)=pl(4s_0+3) +1$ so $pl(4s_0+3)+2 \leq pl(4s_0-5)+2=  pl(4s_0+3)$, contradiction.\\

In conclusion, for all possible cases $w_b(4s_0,4)$ is of type $xyyx$.

\end{proof}

\begin{proposition}
Let $w$ be a finite word of length $4^k$, such that its palindromic length sequence coincides with a prefix of $(pl(n))_{n \in \mathbf{N}}$, then $w$ is a prefix of a sequence in $\mathcal{C}$.
\end{proposition}

\begin{proof}

Let us prove it by induction. The statement is trivially true when $k=0$. Now suppose the statement is true for $k=s_0$, let us consider the case $k=s_0+1$:\\

Remarking that Lemma 2, 3, 4, 5 work under the weaker condition of sequences announced as in previous proposition, we can apply the same results to prove each prefix of $w$ of length $4k$ admits an optimal palindromic decomposition of type $p_1p_2..p_r$ such that the length of all this factors are multiples of $4$. Using Lemma 7 there is another alphabet $\Sigma_1$ and a bijection $f: \Sigma^4 \to \Sigma_1$ such that $f(w)$ is still a word which palindromic length sequence coincides with a prefix of $(pl(n))_{n \in \mathbf{N}}$, however the length of $f(w)$ is $4^{s_0}$, using the hypothesis of induction, it is a prefix of a sequence in $\mathcal{C}$, so $w$ is also a prefix of a sequence in $\mathcal{C}$, by applying the inverse of $f$.

\end{proof}

\begin{theorem}
All sequences such that their palindromic length sequences coincide with the one of Thue-Morse's are in $\mathcal{C}$.
\end{theorem}

\bibliographystyle{alpha}
\bibliography{citations}

\end{document}